\documentclass[12pt]{article}
\usepackage{amsfonts,amsmath,amssymb}
\usepackage{slashed,setspace}
\usepackage[all]{xy}

\usepackage{fullpage}

\usepackage{enumerate}
\usepackage[active]{srcltx}
\def\leftnote#1{\vadjust{\setbox1=\vtop{\hsize 30mm\parindent=0pt\bf\baselineskip=9pt\rightskip=4mm plus 4mm#1}\hbox{\kern-3cm\smash{\box1}}}}

\newtheorem{theorem}{Theorem}[section]

\newtheorem{lemma}[theorem]{Lemma}
\newtheorem{remark}[theorem]{Remark}

\newtheorem{proposition}[theorem]{Proposition}

\newtheorem{notation}[theorem]{Notation}

\newenvironment{proof}[1][Proof]{\par\addvspace{2mm}\noindent\textbf{#1.} }{\ \rule{0.5em}{0.5em}\par\vspace{4mm}}
\newcommand{\Q}{\mathbb{Q}}
\newcommand{\bo}{{\mathfrak O}}
\newcommand{\D}{{\mathfrak D}}

\newcommand{\bp}{{\mathfrak P}}

\newcommand{\G}{\mathcal{G}}

\newcommand{\K}{\mathcal{K}}

\newcommand{\J}{\mathfrak{J}}

\DeclareMathOperator{\coker}{coker}
\DeclareMathOperator{\Gal}{Gal}

\DeclareMathOperator{\Hom}{Hom}
\begin{document}

 \bibliographystyle{plain}
\title{Construction of Self-Dual Integral Normal Bases in Abelian Extensions of Finite and Local Fields}
\author{Erik Jarl Pickett}
%\address{Math\'ematiques, \'Ecole Polytechnique F\'ed\'erale de Lausanne, 
%1015 Lausanne, Switzerland, 
%erikjarl.pickett@epfl.ch}

\maketitle

\begin{abstract}Let $F/E$ be a finite Galois extension of fields with abelian Galois group $\Gamma$. A self-dual normal basis for $F/E$ is a normal basis with the additional property that $Tr_{F/E}(g(x),h(x))=\delta_{g,h}$ for $g,h\in\Gamma$. Bayer-Fluckiger and Lenstra have shown that when $char(E)\neq 2$, then $F$ admits a self-dual normal basis if and only if $[F:E]$ is odd. If $F/E$ is an extension of finite fields and $char(E)=2$, then $F$ admits a self-dual normal basis if and only if the exponent of $\Gamma$ is not divisible by $4$. In this paper we construct self-dual normal basis generators for finite extensions of finite fields whenever they exist.

Now let $K$ be a finite extension of $\Q_p$, let $L/K$ be a finite abelian Galois extension of odd degree 
and let $\bo_L$ be the valuation ring of $L$. We define $A_{L/K}$ to be the unique
fractional $\bo_L$-ideal with square equal to the inverse different
of $L/K$. It is known that a self-dual integral normal basis exists for $A_{L/K}$ if and only if $L/K$ is weakly ramified. Assuming $p\neq 2$, we construct such bases whenever they exist. \end{abstract}

%\keywords{Self-dual normal basis; finite fields; local fields; square-root of the inverse different; Galois module}
%\ccode{Primary 11T30, 11S20, 11S23, 11S31; Secondary 11R33, 12F10, 20C10}

\section{Introduction}
Let $F/E$ be a finite Galois extension of fields with abelian Galois group $\Gamma$. For $x\in F$ we let $Tr_{F/E}(x)=Tr_{\Gamma}(x)=\sum_{\gamma\in\Gamma}{\gamma}(x)$ be the trace of $x$ in $E$. We define $T_{F/E}:F\times F\rightarrow E$ to be the symmetric
non-degenerate $E$-bilinear form associated to the trace map (i.e.,
$T_{F/E}(x,y)=Tr_{F/E}(xy)$) which is $\Gamma$-invariant in the sense
that $T_{L/K}({\gamma}(x),{\gamma}(y))=T_{L/K}(x,y)$ for all $\gamma$ in $\Gamma$.

For $x\in F$ we say that $x$ is a normal basis generator for $F$ over $E$ if $x$ is a generator of $F$ as a $E[\Gamma]$-module. If, in addition, $T_{L/K}(x,{\gamma}(x))=\delta_{1,\gamma}$ for all $\gamma\in\Gamma$, we say that $x$ is a self-dual normal basis generator for $F$ over $E$.
In \cite{Bayer-Lenstra}, Bayer-Fluckiger and Lenstra prove that when $char(E)\neq 2$, then $F$ admits a self-dual normal basis if and only if $[F:E]$ is odd. If $F/E$ is an extension of finite fields and $char(E)=2$, then $F$ admits a self-dual normal basis if and only if the exponent of $\Gamma$ is not divisible by $4$. We remark that if $x\in F$ is a self-dual element, i.e., $T_{L/K}(x,{\gamma}(x))=\delta_{1,\gamma}$ for all $\gamma\in\Gamma$, then $x$ is necessarily a normal basis generator for $F$ over $E$.

In this paper we give explicit constructions for self-dual normal bases at a field level for extensions of finite fields $\mathbb{F}_{q^n}/\mathbb{F}_q$. We also give explicit constructions at an integral level for abelian extensions of local fields $L/K$, giving generators of the unique Galois module $A_{L/K}$ whose square is equal to the inverse different. As these two subjects may be of interest to people with very different backgrounds we have tried to treat the two cases as separately as possible. However, some crossover is unavoidable.

Aside from their intrinsic interest, self-dual normal bases for extensions of finite fields are of use in encryption and have been used by Wang for constructing the Massey-Omura finite-field multiplier (see \cite{Wang}). There are results in the literature constructing such bases, see \cite{Gao}, \cite{Nogami_et_al} and \cite{Wang} but all put some restrictions on the degree of the extension or the characteristic of the base field. We present here constructions for self-dual normal bases for any extension of finite fields for which they exist. The method we use for these constructions is to describe normal basis generators and then modify them slightly to give self-dual elements. We use results of Semaev, \cite{Semaev}, to describe these normal bases, however there is a large amount of literature concerning normal basis generators for finite fields, see, \cite{Lenstra}, \cite{Lenstra-Schoof}, \cite{Sidel'nikov}, \cite{Stepanov-Shparlinsky} and \cite{von_zur_Gathen-Giesbrecht}.

We now assume $K$ to be a finite extension of $\Q_p$ and let $L/K$ be a finite Galois extension. Let $\D_{L/K}$ be the different of the extension $L/K$. When $\D_{L/K}$ has an even valuation we define $A_{L/K}$ to be the unique fractional ideal such that \[A_{L/K}=\D_{L/K}^{-1/2}\] and remark that $A_{L/K}=\bo_L$ when $L/K$ is unramified. The Galois module, $A_{L/K}$, is self-dual with respect to the trace form, i.e., \[A_{L/K}^D=\{a\in L:T_{L/K}(a,A_{L/K})\subseteq\bo_L\}=A_{L/K}.\] 
For $[L:K]$ odd, $\D_{L/K}$ always has an even valuation (see \cite{Serre}, IV Prop. 4) and Erez has proved, in \cite{Erez2}, that an integral normal basis for $A_{L/K}$ exists if and only if $L/K$ is at most weakly ramified. In other words, there exists an $x\in L$ such that $x$ generates $A_{L/K}$ as an $\bo_K [G]$-module if and only if the second ramification group of $L/K$ is trivial. Inspired by the results of Erez and Bayer-Fluckiger and Lenstra, we ask two questions about self-dual integral normal bases for $A_{L/K}$. Firstly, when do such bases exist? Secondly, when they do exist, can we describe them explicitly? 

The first question is now completely answered. With $L/K$ as above, no such basis will exist if $[L:K]$ is even (\cite{Bayer-Lenstra} Theorem 6.1a) and if $[L:K]$ is odd then such a basis exists if and only if $L/K$ is at most weakly ramified (\cite{Fainsilber+Morales}, Corollary 4.8 and \cite{Erez2}, Theorem $1$). 

There are only two results on the explicit construction of self-dual normal bases of $A_{L/K}$ known to us. The first is due to Erez, with $L/\Q_p$ and $L$ contained in some cyclotomic extension of $\Q_p$, \cite{Erez}. The second is due to the author and is a generalisation of Erez's results using Dwork's power series in extensions of local fields generated by Lubin-Tate formal groups, \cite{Pickett}. Specifically, the author studies cyclic degree $p$ extensions, $M/K$, of local fields where $K/\Q_p$ is unramified. Dwork's exponential power series can be used to explicitly describe generators of Kummer extensions, $L/K(\zeta_p)$, with $M\subset L$. If $x\in L$ is such a Kummer generator, then $(1+Tr_{L/M}(x))/p$ is a self-dual normal basis generator for $A_{M/K}$. In comparison, the main method used in this paper is to construct a normal basis generator for $A_{L/K}$, with $L/K$ abelian, and then use Galois resolvends to modify it to give a self-dual normal basis generator. We remark that this method is not needed in the tamely ramified case as then $L/K$ will be Kummer and such a basis can be described directly in terms of generators of $L/K$.

We use these methods to give explicit constructions for self-dual integral normal bases for $A_{L/K}$ whenever $L/K$ is abelian and $p\neq 2$. We remark that some of our constructions still work for $p=2$ and the only case needed for completeness is $L/K$ unramified with $[L:K]=r^i$ where $r$ is an odd prime and $i\in\mathbb{N}$. We should also remark that the constructions in \cite{Erez} and \cite{Pickett} are probably a lot more useful in terms of calculations of invariants such as resolvents and Galois Gauss sums.

This paper is structured into three sections. In Section \ref{key} we prove some key results that describe how to modify normal basis generators so that they generate self-dual normal bases. We give statements for both local and finite fields, but the methods of proof are very similar. In Section \ref{finite} we study extensions of finite fields of characteristic $p$. We reduce the problem to that of studying extensions with degree coprime to $p$ and extensions with degree a $p$th power. We then give constructions for self-dual normal bases for finite extensions of finite fields whenever they exist. In Section \ref{local} we study abelian extensions of local fields $L/K$ where $K$ is an extension of $\Q_p$. We reduce the problem to the separate cases of totally ramified extensions and unramified extensions. We then construct self-dual integral normal bases for $A_{L/K}$; whenever they exist in the totally ramified case; with the assumption that $p\neq 2$ in the unramified case.

\section{Key Results}\label{key}

Let $p$ be an odd prime. Let $K$ be a finite extension of $\Q_p$ with valuation ring $\bo_K$ and residue field $k$. Let $q$ be some power of $p$ and let $\mathbb{F}_q$ be the finite field of $q$ elements. Let $G$ be a finite abelian group. We begin by proving a number of results which will be key to constructing self-dual normal basis generators for extensions of $K$ or $\mathbb{F}_q$ with Galois group isomorphic to $G$. Most of the results for finite fields follow immediately by considering the residue fields in the corresponding result for local fields. We state the results separately to make the theory for finite fields complete in its own right.

Since $G$ is a finite abelian group we can let
$G=G_p\times H$ where $G_p$ is the $p$-primary component of $G$ and
$p\slashed{|}|H|$.

\begin{lemma}Let $\J(\bo_K[G])$ (resp. $\J(\mathbb{F}_q[G])$) be the Jacobson radical of $\bo_K [G]$ (resp. $\mathbb{F}_q[G]$) and $k$
be the residue field of $\bo_K$, we then have the the following exact
sequences.\label{exact sequence}
\begin{equation*}1\rightarrow1+\J(\bo_K[G])\rightarrow\bo_K [G]^{\times}\rightarrow
k[H]^{\times}\rightarrow 1\label{sequence}\end{equation*}\begin{equation*}1\rightarrow1+\J(\mathbb{F}_q[G])\rightarrow\mathbb{F}_q[G]^{\times}\rightarrow
\mathbb{F}_q[H]^{\times}\rightarrow 1.\label{sequence2}\end{equation*}
\end{lemma}

\begin{proof}Let $\pi$ be a uniformising parameter for $\bo_K$. We define a homomorphism of groups
$\theta:\bo_K [G]^{\times}\rightarrow k[H]^{\times}$
as the map generated by sending $g$ to $1$ for all $g\in G_p$ then
reducing modulo $\pi$.

We observe that
$\theta:\bo_K[H]^{\times}\twoheadrightarrow
k[H]^{\times}$ and that $\bo_K[H]^{\times}\subseteq(\bo_K
[G])^{\times}$, and so $\theta$ is surjective.

From \cite{Karpilovsky} Theorem 1, we know that the ideal $\J(\bo_K[G])$ will
be generated by the set $\{\pi,\{g-1\}_{g\in G_p}\}$, therefore we
clearly have $(1+\J(\bo_K[G]))\subseteq \ker\theta$.

Let $\sum_{g\in G_p}\sum_{h\in H}a_{gh}gh$ be an element of
$\ker\theta$, i.e., $\sum_{g\in G_p}\sum_{h\in
H}a_{gh}h=1+a_{\pi}{\pi}$ for some $a_{\pi}\in\bo_k [G]$. We see
that
\[\begin{array}{ll}\displaystyle\sum_{g\in G_p}\sum_{h\in H}a_{gh}gh
&=\displaystyle\sum_{g\in G_p}\sum_{h\in H}a_{gh}(g-1)h+\sum_{g\in G_p}\sum_{h\in H}a_{gh}h\\
&\displaystyle=\sum_{g\in G_p}\sum_{h\in
H}a_{gh}(g-1)h+a_{\pi}{\pi}+1\in(1+\J(\bo_K[G])),
\end{array}\]
and so $\ker\theta\subseteq (1+\J(\bo_K[G]))$.

From \cite{Karpilovsky} Theorem 1, we have that $\J(\mathbb{F}_q[G])$ will
be generated by the set $\{g-1\}_{g\in G_p}$. If we let $K$ be such that $k\cong \mathbb{F}_q$, the exactness of the second sequence now follows from the first by factoring through by $\pi$.
\end{proof}

For any field $F$ we let $J_{F[G]}:F[G]\rightarrow F[G]$ be the $F$-linear group algebra involution defined by
the $F$-linear extension of $J_{F[G]}(g)=g^{-1}$ for $g\in G$. When there is no confusion we shall refer to $J_{F[G]}$ simply as $J$ and for ease of notation we will usually denote $J(a)$ by $\bar{a}$. We now define $\G=\{1,J\}$ as the cyclic group of order $2$
generated by the group algebra involution.

\begin{lemma} \label{cohomology isomorphism lemma}
Using Tate cohomology, the augmentation map on $\bo_K[G]$ (resp. $\mathbb{F}_q[G]$) induces the following isomorphisms:
\[\hat{H}^0(\G,\bo_K[G]^{\times})\cong\hat{H}^0(\G,\bo_K^{\times})\]
\[\hat{H}^0(\G,\mathbb{F}_q[G]^{\times})\cong\hat{H}^0(\G,\mathbb{F}_q^{\times}).\]
\end{lemma}

\begin{proof} We only prove the first result as the second follows immediately by letting $K$ be such that $k\cong \mathbb{F}_q$ and factoring through by $\pi$ at every step.

For ease of notation we let $\J(\bo_K[G])=\J$. From the short exact sequence in Lemma \ref{exact sequence} we
obtain the following long exact sequence (see \cite{Brown} VI \S5
for details).
\begin{eqnarray*}\cdots\rightarrow\hat{H}^{0}(\G,(1+\J))\rightarrow&\hat{H}^{0}(\G,\bo_K
[G]^{\times})&
\rightarrow\hat{H}^{0}(\G,k[H]^{\times})\rightarrow\\&\hat{H}^{1}(\G,(1+\J))&\rightarrow\hat{H}^{1}(\G,\bo_K
[G]^{\times})\rightarrow\cdots\end{eqnarray*}

We restrict the norm map associated to $\G$ to the map
$N^{*}:(1+\J)_{\G}\rightarrow (1+\J)^{\G}$, where
$(1+\J)_{\G}=(1+\J)/I_{\G}(1+\J)=(1+\J)/\{\overline{a}a^{-1}:a\in(1+\J)\}.$
Then, by definition, $\hat{H}^0(\G,(1+\J))=\coker(N^{*})$ and
$\hat{H}_0(\G,(1+\J))=\ker(N^{*}).$ See
\cite{Brown} VI \S4 for full details.

We see that $\coker(N^{*})=(1+\J)^{\G}/N(1+\J)$. If
$x\in(1+\J)^{\G}$ then $N(x)=x^2$, and so $((1+\J)^{\G})^2\subseteq
N(1+\J)$. We know that $(1+\J)$ is a pro-$p$ group and $(p,2)=1$, therefore, from \cite{Ribes-Zalesskii} IV Proposition 4.2.1, $\psi^2:a\rightarrow a^2$ for $a\in 1+\J$ is a group automorphism of $1+\J$. As $(1+\J)=(1+\J)^2$, we know $(1+\J)^{\G}=((1+\J)^2)^{\G}$. If we let $(1+j)^2\in((1+\J)^2)^{\G}$ we have $\overline{(1+j)^2}=(1+j)^2$. The group $G$ is abelian so $J$ is a group automorphism on $(1+\J)$, therefore $\overline{(1+j)}^2=\overline{(1+j)^2}\ (=(1+j)^2)$ and as $\psi^2$ is an automorphism of $1+\J$, we have $1+j=\overline{1+j}$. We then have $(1+\J)^{\G}=((1+\J)^{\G})^2$ which implies that $(1+\J)^{\G}\subseteq N(1+\J)$. We
therefore see that $\hat{H}^0(\G,(1+\J))=\{1\}$.

Now let $x\in \ker(N^{*})$, this means $N(x)=\overline{x}x=1$ and
therefore $\overline{x}=x^{-1}$. As above, we can write $x^{-1}=b^2$ for some $b\in(1+\J)$. Therefore
$x=\overline{b}b^{-1}$, so
$x\in\{\overline{a}a^{-1}:a\in(1+\J)\}=I_{\G}(1+\J)$. Therefore, if $x\in \ker(N^{*})$ then
it is represented by $1$ in $(1+\J)_{\G}$ so we must have
$\hat{H}_0(\G,(1+\J))=\{1\}$.

The group $\G$ is a finite cyclic group so from \cite{Brown} III
\S1, we know that $\hat{H}^i(\G,(1+\J))=\hat{H}_{0}(\G,(1+\J))$ for
all odd $i$. Therefore, $\hat{H}^{1}(\G,(1+\J))$ and
$\hat{H}^{0}(\G,(1+\J))$ are trivial and the long exact sequence
above shows us that we have an isomorphism
$\hat{H}^0(\G,(\bo_K[G])^{\times})\cong\hat{H}^0(\G,(k[H])^{\times})$
induced by the map $\theta$. As a result we observe that
the prime-to-$p$ projection map will induce an isomorphism
$\hat{H}^0(\G,(\bo_K[G])^{\times})\cong
\hat{H}^0(\G,(\bo_K[H])^{\times})$.

Recall that the group $H$ is abelian, of odd order and $char(k)\slashed{|}|H|$.
We can decompose $K[H]$ as \[K[H]\cong K\times L_1 \times
L_2\times\cdots\times L_n\times(L'_1\times L'_1)\times(L'_2\times
L'_2)\times\cdots\times(L'_m\times L'_m) \] where the $L_i$ and
$L'_j$ are all unramified extensions of $K$ and the coefficient in
$K$ is obtained by the map $\varepsilon_H:K[H]\rightarrow K$ given by $\varepsilon_H(h)=1$ for $h\in H$. The group algebra
involution, $J$, will be an element of $\Gal(L_i/K)$ and will act on
the $(L'_j\times L'_j)$ by switching the components.

As $|H|$ is a unit in $\bo_K$ we know that $\bo_K[H]$ is the unique maximal
order in $K[H]$, see \cite{Reiner} Chapter 9, \S41 Theorem 41.1.
Therefore,
\[\bo_K[H]\cong \bo_K\times \bo_{L_1} \times
\bo_{L_2}\times\cdots\times \bo_{L_n}\times(\bo_{L'_1}\times
\bo_{L'_1})\times(\bo_{L'_2}\times
\bo_{L'_2})\times\cdots\times(\bo_{L'_m}\times \bo_{L'_m}) .\]

If $(u,v)\in((\bo_{F_i}\times\bo_{F_i})^{\times})^{\G}$ then
$(u,v)=(v,u)$ so $u=v$. We then have
$(u,u)=(1,u)(u,1)=(1,u)\overline{(1,u)}$ so
$\hat{H}^0(\G,(\bo_{L_i}\times\bo_{L_i})^{\times})=\{1\}$.

Standard theory also gives us
$\hat{H}^0(Gal(L_i/K),\bo_{L_i}^{\times})=\bo_K^{\times}/N(\bo_{L_i}^{\times})=k^{\times}/N(l_i^{\times})=\{1\}$, where $l_i$ is the residue field of $L_i$ (see for example, \cite{Serre} V \S2 Prop 3 and XIII \S2). Therefore,
the map $\varepsilon_H:K[H]\rightarrow K$ induces an isomorphism
$\hat{H}^0(\G,(\bo_K[H])^{\times})\cong
\hat{H}^0(\G,\bo_K^{\times})$. Combining this with the isomorphism
$\hat{H}^0(\G,(\bo_K[G])^{\times})\cong
\hat{H}^0(\G,(\bo_K[H])^{\times})$ we see that the augmentation map
on $G$ induces the desired isomorphism.
\end{proof}

Let $F/E$ be a Galois extension of fields and let $a\in F$. We define the resolvend of $a$ as
\[r_{F/E}(a)=\sum_{g\in \Gal{F/E}}g(a)g^{-1},\]
and let $R_{F/E}(a)$ be defined as
\[R_{F/E}(a)=r(a)\overline{r(a)}=\sum_{g\in
\Gal(F/E)}T_{F/E}(a,g(a))g.\]

We now let $L/K$ and $\mathbb{F}_{q^n}/\mathbb{F}_q$ be Galois abelian extensions of fields and let $G=\Gal(L/K)$ and $\tilde{G}=\Gal(\mathbb{F}_{q^n}/\mathbb{F}_q)$.

\begin{lemma}\begin{enumerate}\item Let $x\in A_{L/K}$ be an integral normal basis generator for $A_{L/K}$, i.e., $A_{L/K}=\bo_K[G].x$ (this means that we suppose $L/K$ is at most weakly ramified). Then \[R_{L/K}(x)\in(\bo_K[G])^{\times}.\]\item Let $y\in \mathbb{F}_{q^n}$ be a normal basis generator for $\mathbb{F}_{q^n}$ over $\mathbb{F}_q$, i.e., $\mathbb{F}_{q^n}=\mathbb{F}_{q}[\tilde{G}].y$. Then \[R_{\mathbb{F}_{q^n}/\mathbb{F}_q}(y)\in(\mathbb{F}_q[\tilde{G}])^{\times}.\]\end{enumerate} \label{unit_lemma}\end{lemma}
\begin{proof} \textit{Part 1.}

By definition, we know that $T_{L/K}:A_{L/K}\times A_{L/K}\longrightarrow\bo_K$, so clearly $R_{L/K}(x)\in\bo_K [G]$.

Let $\rho$ be the regular representation of the group $G$. For $g\in
G$, the matrix $\rho(g)$ will have a $1$ in the position $(h,gh)$
for all $h\in G$ and a $0$ in every other position. As the trace
form is a $G$-invariant symmetric form, and as $G$ is abelian we
know that
\[ T_{L/K}(g(x),x)=T_{L/K}(x,g(x))=T_{L/K}(h(x),hg(x))=T_{L/K}(h(x),gh(x))\] for all $g,h\in G$.
If we use $g\in G$ as a basis for $\bo_K[G]$ we see that the matrix
$\rho(R_{L/K}(x))$, representing $R_{L/K}(x)=\sum_{g\in G}T_{L/K}(g(x),x)g$ must
have the coefficient $T_{L/K}(g(x),x)=T_{L/K}(h(x),gh(x))$ in the
position $(h,gh)$ for all $g,h\in G$. Therefore, after relabeling,
$\rho(R_{L/K}(x))$ is represented by the matrix $(T_{L/K}(g(x),h(x)))$.

 Since $\{g(x):g\in G\}$ is a basis for $A_{L/K}$ over $\bo_K$, we know
that \[\det(T_{ L/K}(g(x),h(x)))\bo_K=[A_{L/K}^D:A_{L/K}]\bo_K\](see
\cite{Frohlich-Taylor} III 2.8), but $A_{L/K}^D=A_{L/K}$, so
$\det(T_{ L/K}(g(x),h(x)))=\det_{\rho}(R_{x})$ must be a unit.

We now let $\hat{G}$ be the group of irreducible characters of $G$
over a fixed algebraic closure of $K$. The regular representation,
$\rho$, is equal to the sum of the irreducible representations so we
know that
\begin{equation}{\rm det}_{\rho}(R_{x})=\prod_{\chi\in\hat{G}}{\rm det}_{\chi}(R_{L/K}(x))=\prod_{\chi\in\hat{G}}\chi(R_{x}).\label{detR_Xeqn}\end{equation}

The group $G$ is finite and abelian so there exists a finite
extension $F/K$, over which the irreducible characters of ${G}$ are
realisable. We then have $F[G]\cong\prod_{\chi\in\hat{G}}F$ and the
isomorphism is given by $\sum_{g\in G}f_gg\mapsto\prod_{\chi\in
\hat{G}} \sum_{g\in G}f_g\chi(g)$. For more details see, for
example, \cite{serrereps} \S6.2.

For each $\chi\in\hat{ G}$ we know that $\chi(R_{x})=\sum_{g\in
G}T_{ L/K}(x,g(x))\chi(g)$, that $\chi(g)\in\bo_F$ (see
\cite{serrereps} 6.5 Prop. $15$) and that $T_{L/K}(x,g(x))$ is
an element of $\bo_K\subseteq\bo_F$. Therefore
$\chi(R_{x})\in\bo_F$, and from Eq. (\ref{detR_Xeqn}) we know
that $\chi(R_{x})$ must be a unit in $\bo_F$.

Under the isomorphism mentioned above $\bo_F[G]\subset
\prod_{\chi\in\hat{G}}\bo_F=\mathfrak{M}$ and $\mathfrak{M}$ is a
maximal order. The element $R_{L/K}(x)$ is represented in $\mathfrak{M}$ as
$\prod_{\chi\in\hat{G}}\chi(R_{L/K}(x))$. We have seen above that
$\chi(R_{L/K}(x))$ is a unit in $\bo_F$ for all $\chi\in\hat{G}$, so $R_{L/K}(x)$
is invertible in $\mathfrak{M}$. As $\bo_F/\bo_K$ is a finite ring
extension we have
$\mathfrak{M}^{\times}\cap\bo_K[G]=\bo_K[G]^{\times}$, see
\cite{Reiner} Ch. 6 \S25. Therefore $R_{L/K}(x)\in\bo_K[G]^{\times}$.

\textit{Part 2.} We now let $K$ be a local field with $k\cong\mathbb{F}_q$. Let $L/K$ be the unique unramified extension of degree $n$. We then have the residue field of $L$, denoted $l$, being isomorphic to $\mathbb{F}_{q^n}$. From \cite{Serre}, IV Prop. 4 we see that $v_{L}(\D_{L/K})=0$ and so $A_{L/K}=\bo_L$. Let $\pi$ be a uniformising parameter of $\bo_K$, then the Jacobson radical of $\bo_K$ is $\pi\bo_K$. By assumption we have $l= k[\tilde{G}].y$ which implies that $\bo_L\cong \bo_K[\tilde{G}].y\mod\pi\bo_L$. Therefore, by Nakayama's lemma (see \cite{Atiyah-MacDonald}, Ch.2 Prop. 2.6), we have $x\in\bo_L$ is an integral normal basis generator for $\bo_L$ if $x\equiv y\mod\pi\bo_L$. Using \textit{part 1.} $R_{L/K}(x)\in(\bo_K[\tilde{G}])^{\times}$ for all $x\equiv y\mod\pi\bo_L$ and therefore $R_{l/k}(y)\in (k[\tilde{G}])^{\times}$.

\

\end{proof}

\begin{notation}For a Galois extension of fields $F/E$ with Galois group $\Gamma$ we define the action $\circ$ of $E[\Gamma]$ on $F$. For $a=\sum_{\gamma\in
\Gamma}a_{\gamma}\gamma\in E[\Gamma]$ and $y\in F$, then \[a\circ y=\sum_{\gamma\in
\Gamma}a_{\gamma}{\gamma}(y).\] We recall that the resolvend of $y$ is
defined as $r_{F/E}(y)=\sum_{\gamma\in \Gamma}{\gamma}(y)\gamma^{-1}$.\end{notation}

 A straightforward
calculation then shows, for $w\in\bo_K[G]$, that $wr(y)=r(w\circ y)$.

\begin{theorem}
 \begin{enumerate}
  \item Let $x\in A_{L/K}$ be an integral normal basis generator for $A_{L/K}$, then $R_{L/K}(x)=u\bar{u}$ for some $u\in\bo_K[G]^{\times}$ and $u^{-1}\circ x$ is a self-dual integral normal basis generator for $A_{L/K}$.
\item Let $y\in \mathbb{F}_{q^n}$ be a normal basis generator for $\mathbb{F}_{q^n}$ over $\mathbb{F}_q$, then $R_{\mathbb{F}_{q^n}/\mathbb{F}_q}(y)=v\bar{v}$ for some $v\in\mathbb{F}_q[\tilde{G}]^{\times}$ and $v^{-1}\circ x$ is a self-dual normal basis generator for $\mathbb{F}_{q^n}$ over $\mathbb{F}_q$.
 \end{enumerate}
\label{construction_theorem}
\end{theorem}

\begin{proof} \textit{Part 1.}
We first note that from Lemma \ref{unit_lemma} we have $R_{L/K}(x)\in\bo_K[G]^{\times}$.
 From Lemma \ref{cohomology isomorphism lemma} we know that the
augmentation map $\varepsilon:g\mapsto 1$ for $g\in G$ gives an
isomorphism
$\hat{H}^0(\G,(\bo_K[G])^{\times})\cong\hat{H}^0(\G,\bo_K^{\times})$.
We know that $\bo_K^{\times}$ is invariant under $\G$, so
$\hat{H}^0(\G,\bo_K^{\times})=\bo_K^{\times}/(\bo_K^{\times})^2$. A
straightforward calculation, left as an exercise, gives
$\varepsilon(R_{L/K}(x))=Tr_{ L/K}(x)^2$. Therefore $R_{L/K}(x)$ is represented by
$1$ in $\hat{H}^0(\G,(\bo_K[G])^{\times})$, and so we can now write
$R_{L/K}(x)=u\bar{u}$ for some fixed $u\in(\bo_K[G])^{\times}$.

We have $R_{L/K}(x))=r_{L/K}(x)\overline{r_{L/K}(x)}=u\bar{u}$ and $L[G]$ is a
commutative ring, so $u^{-1}r_{L/K}(x)\overline{u^{-1}r_{L/K}(x)}=1$. Therefore,
$r_{L/K}(u^{-1}\circ x)\overline{r_{L/K}(u^{-1}\circ x)}=1$. We now let
$x'=u^{-1}\circ x$ and note that $R_{L/K}(x')=\sum_{g\in
G}T_{L/K}(x',g(x'))g=1$. This means that
$T_{L/K}(x',g(x'))=\delta_{1,g}$ for $g\in G$ and, as the trace form
is $G$-invariant, $T_{L/K}(g(x'),h(x'))=\delta_{g,h}$ for $g,h\in
G$.

Finally, as $A_{L/K}=\bo_K[G].x$, clearly $x'\in A_{L/K}$.
Therefore, applying \cite{Pickett} Lemma 8, we see that $x'$ is
an integral self-dual normal basis element for $A_{L/K}$.

 \textit{Part 2.}  We first note that that any element $y'\in\mathbb{F}_{q^n}$ with the property $T_{\mathbb{F}_{q^n}/\mathbb{F}_q}(y',\tilde{g}(y'))=\delta_{1,\tilde{g}}$ for $\tilde{g}\in\tilde{G}$ is necessarily a normal basis generator for $\mathbb{F}_{q^n}$ over $\mathbb{F}_q$. The result then follows from exactly the same arguments as \textit{part 1.} but using the corresponding statements for finite fields in Lemmas \ref{cohomology isomorphism lemma} and \ref{unit_lemma}.

\end{proof}

We remark that Theorem \ref{construction_theorem} gives an alternative and constructive proof of the existence of a self-dual integral normal basis for both $A_{L/K}$ with $L/K$ odd, weakly ramified and abelian; and for $\mathbb{F}_{q^n}$ over $\mathbb{F}_q$ for $q$ and $n$ odd.

The method we shall use for the constructions in this paper (with the exception of the totally, tamely ramified case for local fields) will be to describe a normal basis generator and then calculate the units $u$ and $v$ in Theorem \ref{construction_theorem}.

\begin{notation}
Let $R$ be a commutative ring with identity and let $r\in R^2$, i.e. $r$ is the square of some element in $R$. Throughout we shall use $\sqrt{r}$ to denote some element in $R$ such that $\left(\sqrt{r}\right)^2=r$. We note that $\sqrt{r}$ need not be unique, but this does not limit the constructions described in the sequel.\label{sqrt_notation}
\end{notation}

\section{Finite Fields}\label{finite}

In this section we describe how to construct self-dual normal bases for any finite extension of a finite field for which they exist.  We first prove a lemma which will be very useful in the sequel. We do not restrict ourselves to finite fields for this result.

\begin{lemma}Let $F_1/E$ and $F_2/E$ be abelian Galois extensions with $F_1\cap F_2=E$ and $F=F_1F_2$. For $i=1,2$ let $\Gamma_i=\Gal(F_i/E)$ and let $x_i\in F_i$ be a self-dual element for $F_i/E$ in the sense that $T_{F/E}({g_i}(x_i),{h_i}(x_i))=\delta_{g_i,h_i}$ for $g_i,h_i\in \Gamma_i$. Then $x_1x_2$ is a self-dual element for $F/E$.\label{mult_lemma}\end{lemma}
\begin{proof}We first observe that, as the trace form is $\Gamma_i$-equivarient, we have $T_{F/E}({g_i}(x_i),{h_i}(x_i))=\delta_{g_i,h_i}$ if and only if $T_{F/E}(x_i,{g_i}(x_i))=\delta_{1,g_i}$.  With a slight abuse of notation we have $\Gal(F/E)=\Gamma_1\times \Gamma_2$ with $F_i$ being the fixed field of $\Gamma_j$ when $i\neq j$. Therefore, with $g_i\in \Gamma_i$ we have \[\begin{array}{ll}T_{F/E}(x_1x_2,{(g_1,g_2)}(x_1x_2))&=T_{F_1/E}(x_1,{g_1}(x_1))T_{F_2/E}(x_2,{g_2}(x_2))\\&=\left\{\begin{array}{ll}1&\text{if\ }g=h=1\\0& \text{otherwise.}\end{array}\right.\end{array}\]\end{proof}

%\begin{lemma} Let $L/K$ be a finite Galois extension of fields with $\Gal(L/K)=G$. Let $x\in L$ such that $Tr_{L/K}(x,x^g)=\delta_{1,g}$ for all $g\in G$. Then $x$ is a normal basis generator for $L/K$\label{in_and_self-dual_field_lemma}\end{lemma}
%\begin{proof} We just need to show that $\{x^g:g\in G\}$ is linearly independent over $K$. For contradiction we assume that for some $h\in G$ there exists $a_g\in K$ such that $\displaystyle x^h=\sum_{g\in G,g\neq h}a_gx^g$. We then have \begin{eqnarray}\nonumber x^hx^h&=&\sum_{g\in G),g\neq h}a_gx^gx^h\\Tr(x^hx^h)&=&\sum_{g\in G,g\neq h}a_gTr(x^gx^h)\nonumber\\1&=&0\nonumber\end{eqnarray} which is a contradiction.\end{proof}

Let $p$ be a prime and $q$ be some power of $p$. Let $\mathbb{F}_q$ be the finite field of $q$ elements and consider a finite extension of $\mathbb{F}_q$ of degree $m$ with Galois group $G$. This extension is uniquely defined, up to isomorphism, by $m$ and is denoted $\mathbb{F}_{q^m}/\mathbb{F}_q$ with $G=\Gal(\mathbb{F}_{q^m}/\mathbb{F}_{q})\cong C_m$ where $C_m$ is the cyclic group of $m$ elements. As discussed in the introduction, if $p\neq 2$, then $\mathbb{F}_{q^m}$ admits a self-dual normal basis if and only if $m$ is odd. If  $p=2$, then $\mathbb{F}_{q^m}$ admits a self-dual normal basis if and only if the exponent of $G$ is not divisible by $4$.
 
 We have $m=p^n.d$ for some $n,d\in\mathbb{Z}^+$ and $\mathbb{F}_{q^m}=\mathbb{F}_{q^d}\mathbb{F}_{q^{p^n}}$ with $\mathbb{F}_{q^d}\cap\mathbb{F}_{q^{p^n}}=\mathbb{F}_q$. Any self-dual element in a field extension is necessarily a self-dual normal basis generator. Therefore, from Lemma \ref{mult_lemma} if we have a self-dual normal basis generator for $\mathbb{F}_{q^d}/\mathbb{F}_q$ and for $\mathbb{F}_{q^{p^n}}/\mathbb{F}_q$ then the product will be a self-dual normal basis generator for $\mathbb{F}_{q^m}/\mathbb{F}_q$. We are thus reduced to finding self-dual normal bases when $m$ is either a power of $p$ or when $m$ is coprime to $p$.

\subsection{$p$-extensions of finite fields}\label{p_finite_section}

We now consider the extension $\mathbb{F}_{q^{p^n}}/\mathbb{F}_q$
of degree $p^n$. In \cite{Semaev} \S4 Semaev explicitly constructs a
normal basis generator for $\mathbb{F}_{q^{p^n}}$ over
$\mathbb{F}_q$. We include an outline here for completeness; for
full details see $\cite{Semaev}$.

The construction makes use of the following lemma.
\begin{lemma}
 An element $\eta\in\mathbb{F}_{q^{p^n}}$ is a normal basis generator over $\mathbb{F}_q$ if and only if
 $Tr_{\mathbb{F}_{q^{p^n}}/\mathbb{F}_q}(\eta)\neq 0$.\label{unram_p_lemma}
\end{lemma}
\begin{proof}
 See \cite{Semaev} \S4 Lemma 4.1
\end{proof}

Assume that $\eta_0$ is an element of $\mathbb{F}_q$ such that
$Tr_{\mathbb{F}_{q}/\mathbb{F}_p}(\eta_0)\neq 0$. For $i\geq 1$ we
define $f_i(X)\in\mathbb{F}_{q^{p^{i-1}}}[X]$ and
$\eta_i\in\mathbb{F}_{q^{p^i}}$ inductively as follows.
\[f_i(X)=X^p-\eta_{i-1}X^{p-1}+\eta_{i-1}^{2p-1}\] and
$\eta_i$ is an arbitrary root of $f_i(X)$. From \cite{Semaev}
\S4 each polynomial $f_i(X)$
is irreducible over $\mathbb{F}_{q^{p^{i-1}}}$ and
$Tr_{\mathbb{F}_{q^{p^i}}/\mathbb{F}_q}(\eta_i)\neq 0$. Therefore,
from Lemma \ref{unram_p_lemma}, $\eta_n$ will generate a normal
basis for $\mathbb{F}_{q^{p^n}}$ over $\mathbb{F}_q$.

For any polynomial $f(X)=\sum_{i}a_iX^i$ we let
$f(X)^{q^{j}}=\sum_{i}a_i^{q^j}X^i$, then the minimal
polynomial of $\eta=\eta_n$ over $\mathbb{F}_q$ is
\[\prod_{j=0}^{p^{n-1}-1}f_n(X)^{q^{j}}\](see \cite{Semaev}
\S4 for full details).

To construct $\eta_0$ we define the finite field $F'$ such that
$\mathbb{F}_p\subseteq F'\subseteq\mathbb{F}_q$, $[\mathbb{F}_q:F']$
 is prime to $p$ and $[F':\mathbb{F}_p]$ is a power of $p$. As above,
  we then construct $\eta'\in F'$ such that $\eta'$ is a normal basis
   generator for $F'$ over $\mathbb{F}_p$. We can then take $\eta_0=\eta'\in\mathbb{F}_q$ as $[\mathbb{F}_q:F']$ is prime to $p$.

\begin{theorem}Let $p$ be an odd prime and $n$ be any integer. Let $\eta$ be a normal basis generator for $\mathbb{F}_{q^{p^n}}$ over $\mathbb{F}_q$ constructed as above and let $C_{p^n}\cong G=\Gal(\mathbb{F}_{q^{p^n}}/\mathbb{F}_q)$. Then $R_{\mathbb{F}_{q^{p^n}}/\mathbb{F}_q}(\eta)$ is a square in $\mathbb{F}_q[G]^{\times}$ and using Notation \ref{sqrt_notation},
then \[\frac{1}{\sqrt{R_{\mathbb{F}_{q^{p^n}}/\mathbb{F}_q}(\eta)}}\circ
\eta\] is a self-dual normal basis generator for
$\mathbb{F}_{q^{p^n}}$ over $\mathbb{F}_q$. \label{finite_p_theorem}
\end{theorem}

\begin{proof}
From Lemma \ref{exact sequence} we have the exact sequence \begin{equation*}1\rightarrow1+\J(\mathbb{F}_q[G])\rightarrow\mathbb{F}_q[G]^{\times}\stackrel{\varepsilon}\rightarrow
\mathbb{F}_q^{\times}\rightarrow 1.\end{equation*} where $\varepsilon$ is the augmentation map and $\J(\mathbb{F}_q[G])$ is the Jacobson radical of $\mathbb{F}_q[G]$. We observe that $\mathbb{F}_q^{\times}\subseteq\mathbb{F}_q[G]^{\times}$ so there exists a natural embedding, $j:\mathbb{F}_q^{\times}\rightarrow\mathbb{F}_q[G]^{\times}$ which has the property $\varepsilon 
j=1_{\mathbb{F}_q^{\times}}$. Therefore the short exact sequence is split and \[\mathbb{F}_q[G]^{\times}=\mathbb{F}_q^{\times}\times1+\J(\mathbb{F}_q[G]).\]
A straightforward calculation, left as an exercise, shows that $\varepsilon(R_{\mathbb{F}_{q^{p^n}}/\mathbb{F}_q}(\eta))=Tr_{\mathbb{F}_{q^{p^n}}/\mathbb{F}_q}(\eta)^2$, and so $\varepsilon(R_{\mathbb{F}_{q^{p^n}}/\mathbb{F}_q}(\eta))$ is a square in $\mathbb{F}_q^{\times}$. The group $1+\J(\mathbb{F}_q[G])$ is a $p$-group with
$(p,2)=1$, so every element of $1+\J(\mathbb{F}_q[G])$ will be a square. Therefore,
we can write $R_{\mathbb{F}_{q^{p^n}}/\mathbb{F}_q}(\eta)=w^2$ for some $w\in\mathbb{F}_q [G]^{\times}$.

From Theorem \ref{construction_theorem} we know that $R_{\mathbb{F}_{q^{p^n}}/\mathbb{F}_q}(\eta)=v\bar{v}$ for some $v\in\mathbb{F}_q[G]^{\times}$.
We then have $w^2=v\overline{v}$ and $\overline{w}^2=\overline{w^2}=\overline{v\overline{v}}=v\overline{v}$, therefore $w^2=\overline{w}^2$. Since $w\in\mathbb{F}_q[G]^{\times}$ we can write $w=(w_1,w_2)$ and $w^2=(w_1^2,w_2^2)$ with
$w_1\in \mathbb{F}_q^{\times}$ and $w_2\in1+\J(\mathbb{F}_q[G])$. We then see that
$\overline{w}=(w_1,\overline{w}_2)$ and $\overline{w}^2=(w_1^2,\overline{w}_2^2)=w^2$. The group $1+\J(\mathbb{F}_q[G])$ is a
$p$-group with $(p,2)=1$ so $\psi_2:a\mapsto a^2$ is an
automorphism of $1+\J(\mathbb{F}_q[G])$. Therefore, if $w_2^2=\overline{w}_2^2$
then $w_2=\overline{w}_2$, and so $w=\overline{w}$. If we then write $w=\sqrt{R_{\mathbb{F}_{q^{p^n}}/\mathbb{F}_q}(\eta)}$ our result follows from Theorem \ref{construction_theorem}.
\end{proof}

\begin{remark} If $\mathbb{F}_q$ is of characteristic $2$, then, from \cite{Bayer-Lenstra} Theorem 6.1, we know that a self-dual normal basis exists for $\mathbb{F}_{q^{p^n}}$ over $\mathbb{F}_q$ if and only if the exponent of $G$ is not divisible by $4$. In the setting of this section that would mean $n=1$. If we let $\xi\in\mathbb{F}_{q^2}$ be such that $Tr_{\mathbb{F}_{q^2}/\mathbb{F}_q}(\xi)=1$ then $Tr(\xi,\xi)=Tr(\xi^2)=Tr(\xi)^2$ and $Tr(\xi,g(\xi))=2N(\xi)=0$ for $g\in\Gal(\mathbb{F}_{q^2}/\mathbb{F}_q)$. Therefore $\xi$ will be a self-dual normal basis generator. \label{remark_R'}
\end{remark}

\subsection{$p'$-extensions of finite fields}

We now consider the extension $\mathbb{F}_{q^d}/\mathbb{F}_q$ with Galois group $G\cong C_d$ where $d=r^i$ is a power of some odd prime $r\neq p$. We begin, after Semaev \cite{Semaev}, by describing how to decompose $\mathbb{F}_{q^d}$ into sub-$\mathbb{F}_q$-vector spaces that are stable under the action of $G$. This decomposition is directly related to the standard decomposition of the group algebra $\mathbb{F}_q[G]$, which we shall use in the latter part of this section. However, since the first aim is to give an expression for a normal basis generator in terms of roots of a polynomial irreducible over $\mathbb{F}_q$, we first need to describe the explicit decomposition of $\mathbb{F}_{q^d}$.

In fact we first need to pass to a slightly bigger field to get the relations we need between $q$ and $r$. 
We let $v$ be the order of $q$ modulo $r$ and let $q_1=q^v$. We remark that $(v,d)=1$ and that if $q\equiv 1\mod r$, then $q=q_1$.

We observe that $\Gal(\mathbb{F}_{q_1^d}/\mathbb{F}_{q_1})$ is generated by the Frobenius element, $\phi_{q_1}$ which acts as $\phi_{q_1}:x\mapsto x^{q_1}$ for $x\in\mathbb{F}_{q_1^d}$. With this in mind we define the equivalence relation, $\sim$, on the set $\{0,1,\ldots,d-1\}$ as
\[s_1\sim s_2 \text{\ \ if and only if\ \ } s_1\equiv s_2 q_1^n\mod d\] for some $n\in\mathbb{Z}$.
 We then let $S_{q_1}$ be a set of representatives for the equivalence classes of $\{0,1,\ldots,d-1\}$ under $\sim$.

Let $\zeta$ be a primitive $(q_1-1)$th root of unity in
$\mathbb{F}_{q_1}$, then the order of $\zeta$ in $\mathbb{F}_{q_1}^{\times}$ is $q_1-1$ and $X^d-\zeta$ is irreducible over $\mathbb{F}_{q_1}$ (see \cite{Semaev} \S2). Let $\theta\in\mathbb{F}_{q_1^d}$ be a root of $X^d-\zeta$, we then have $\mathbb{F}_{q_1^d}=\mathbb{F}_{q_1}(\theta)$ and we will now describe a normal basis generator of $\mathbb{F}_{q_1^d}/\mathbb{F}_{q_1}$ in terms of $\theta$. Let $V_s$ be the sub-vector space of $\mathbb{F}_{q_1^d}$ spanned over $\mathbb{F}_{q_1}$ by the elements of the set  $\{\theta^t:0\leq t\leq d-1,t\sim s\}$. We note that $V_s\cap V_t=\mathbb{F}_{q_1}$ if and only if $s\slashed{\sim}t$ and that $\phi_{q_1}(V_s)=V_s$. From \cite{Semaev} \S2 we then have
\[\mathbb{F}_{q_1^d}=\prod_{s\in S_{q_1}} V_s.\] From
\cite{Semaev}, Lemma 2.1 we know that for
$\xi\in\mathbb{F}_{q_1}^d$, then $\xi$ is a normal basis generator for
$\mathbb{F}_{q_1^d}/\mathbb{F}_{q_1}$ if and only if $\xi$ has a non-zero component in each $V_s$ under this decomposition. In particular, we note that $\xi=\sum_{s\in S_{q_1}}\theta^s$ has this property. If $\xi$ is such an element, from
\cite{Semaev} Theorem 2.4 we then know that
$Tr_{\mathbb{F}_{q_1^d}/\mathbb{F}_{q^d}}(\xi)$ will be a normal
basis generator for $\mathbb{F}_{q^d}/\mathbb{F}_q$. Therefore,
  \[\eta=Tr_{\mathbb{F}_{q_1^d}/\mathbb{F}_{q^d}}(\sum_{s\in S_{q_1}}\theta^s)\] is a normal basis generator for $\mathbb{F}_{q^d}/\mathbb{F}_q$.

%From \cite{Semaev} \S2 we know that if $a\in\mathbb{F}_{q_1}^*$ such that $(d,(q_1-1)/|a|)=1$, where $|a|$ is the order of $a$ in $\mathbb{F}_{q_1}^*$, then $X^d-a$ is irreducible over
%$\mathbb{F}_{q_1}$. If we let $\theta$ be a
 %root of $X^d-a$, then $\sum_{s\in S_{q_1}}\theta^s$ will be a normal basis generator for $\mathbb{F}_{q_1^d}/\mathbb{F}_{q_1}$
 %(see \cite{Semaev} Lemma 2.1 and Theorem 2.2). In particular, let $\zeta$ be a primitive $(q_1^d-1)$th root of unity,
  %we can then take $a=\zeta^d$ and $\theta=\zeta$. From \cite{Semaev} Theorem 2.4 we then have
  %\[\eta=Tr_{\mathbb{F}_{q_1^d}/\mathbb{F}_{q^d}}(\sum_{s\in S_{q_1}}\zeta^s)\] is a normal basis generator for $\mathbb{F}_{q^d}/\mathbb{F}_q$.

From Theorem \ref{construction_theorem} we know that $R_{\mathbb{F}_{q^{d}}/\mathbb{F}_q}(\eta)=v\bar{v}$ for some $v\in\mathbb{F}_q[G]^{\times}$ and that $v^{-1}\circ\eta$ is a self-dual normal basis generator. Our aim now is to construct $v$. 

The group $G$ is a finite cyclic group of exponent $d=r^i$ and the characteristic of $\mathbb{F}_q$
does not divide $d$. We let $\zeta_d$ be a fixed primitive $d$th root of unity and let $\tilde{\mathbb{F}_q}=\mathbb{F}_q(\zeta_d)$.
Let $\hat{G}=\Hom(G,\tilde{\mathbb{F}}_q^{\times})$ be the character group of $G$ over $\tilde{\mathbb{F}_q}$ and for $\chi\in\hat{G}$
define $\mathbb{F}_q(\chi)$ as the field obtained from $\mathbb{F}_q$ by adjoining the values $\chi(g)$ for all $g\in G$. Let $g$ be a fixed generator of $G$ and define $\chi_i\in\hat{G}$ as $\chi_i:g\mapsto \zeta_d^i$ for $i\in\{0,1,\ldots,d-1\}$.
  From \cite{Karpilovsky_Unit} Ch.4 Theorem 1.4 the map $\chi_s:\mathbb{F}_q[G]\rightarrow \mathbb{F}_q(\chi_s)$ induces a decomposition  isomorphism
   \[D:\mathbb{F}_q[G]\cong \prod_{s\in S}\mathbb{F}_q(\chi_s)\] where $S=S_q$ is the set of representatives of equivalence classes described above.

 Let $DR$ and $Dv$ be the images of $R_{\mathbb{F}_q^d/\mathbb{F}_q}(\eta)$ and $v$ respectively under this isomorphism. We then have $Dv=\prod_{s\in S}v_s$ for some $v_s\in \mathbb{F}_q(\chi_s)$ and $DR=Dv\overline{Dv}=\prod_{s\in S}v_s\overline{v_s}$. To calculate the image of an element under the inverse of the decomposition map $D$, it is only required to solve a
set of $|S|$ simultaneous linear equations. Therefore, we are just left with the task of calculating the action of $J$ on $\mathbb{F}_q(\chi_s)$ and finding $v_s\in \mathbb{F}_q(\chi_s)^{\times}$ such that $\chi_s(R_{\mathbb{F}_{q^d}/\mathbb{F}_q}(\eta))=v_s\overline{v_s}$ for each $s\in S$. In the following three lemmas we describe how to calculate $v_s$ for all $s\in S$.

\begin{lemma}
 \[v_0=Tr_{\mathbb{F}_{q^d}/\mathbb{F}_q}(\eta).\]\label{0_lemma}
\end{lemma}
\begin{proof}
 We have $\mathbb{F}_q(\chi_0)=\mathbb{F}_q$ and $J$ will act by
fixing the whole field. The character $\chi_0$ is actually
 the same map as the augmentation map $\varepsilon$, so we know that $\chi_0(R_{\mathbb{F}_{q^d}/\mathbb{F}_q}(\eta))=(Tr_{{\mathbb{F}_{q^d}/\mathbb{F}_q}}(\eta))^2$.
  Therefore, with $v_0=\overline{v_0}=Tr_{{\mathbb{F}_{q^d}/\mathbb{F}_q}}(\eta)$ we have $\chi_0(R_{\mathbb{F}_{q^d}/\mathbb{F}_q}(\eta))=v_0\overline{v_0}(=v_0^2)$.
\end{proof}

\ 

 For $0\neq s\in S$, we define $s'$ to be the representative of $(d-s)$ in $S$.

\begin{lemma}
Let $s\in S$ be such that $s\neq s'$, then \[v_s=\chi_s(r_{\mathbb{F}_{q^d}/\mathbb{F}_q}(\eta)) \text{\ \ and\ \ } v_{s'}=1 \text{\ \ \ \ for $s<s'$.}\] (Recall that $S$ can be thought of as a subset of $\{0,\ldots,d-1\}$ so that $<$ is defined.)\label{sns'_lemma}
\end{lemma}
\begin{proof}
 Assume $s\neq s'$, then $J$ acts by switching the components of $\mathbb{F}_q(\chi_s)$ and $\mathbb{F}_q(\chi_{s'})$ in the decomposition. We know that  $R_{\mathbb{F}_{q^d}/\mathbb{F}_q}(\eta)$ is fixed by $J$, so we must have $\chi_s(R_{\mathbb{F}_{q^d}/\mathbb{F}_q}(\eta))=\chi_{s'}(R_{\mathbb{F}_{q^d}/\mathbb{F}_q}(\eta))$.
 Therefore, for $s<s'$, if we let $v_s=\chi_s(R_{\mathbb{F}_{q^d}/\mathbb{F}_q}\eta)$ and $v_{s'}=1$, then
 $(\chi_s(R_{\mathbb{F}_{q^d}/\mathbb{F}_q}(\eta)),\chi_{s'}(R_{\mathbb{F}_{q^d}/\mathbb{F}_q}(\eta)))=(v_s,v_{s'})\overline{(v_s,v_{s'})}$.

\

\end{proof}

If $s=s'$ we remark that $\G$ will be a subgroup of $\Gal(\mathbb{F}_q(\chi_s)/\mathbb{F}_q)$, and so $\Gal(\mathbb{F}_q(\chi_s)/\mathbb{F}_q)$ will have an even order. Therefore any element of $\mathbb{F}_q(\chi_s)^{\G}$, i.e., any element of $\mathbb{F}_q(\chi_s)$ that is fixed by $\G$, is a square in $\mathbb{F}_q(\chi_s)$.

When $s=s'$ we will need to give different values of $v_s$ depending on whether $\chi_s(R_{\mathbb{F}_{q^d}/\mathbb{F}_q}(\eta))$ and/or $-\chi_s(R_{\mathbb{F}_{q^d}/\mathbb{F}_q}(\eta))$ are squares in $ \mathbb{F}_q(\chi_s)^{\G}$.

\begin{lemma}
Let $s\in S$ be such that $s=s'$ and recall Notation \ref{sqrt_notation}. \begin{enumerate}\item If $\chi_s(R_{\mathbb{F}_{q^d}/\mathbb{F}_q}(\eta))$ is a square in $ \mathbb{F}_q(\chi_s)^{\G}$, then \[v_s=\sqrt{\chi_s(R_{\mathbb{F}_{q^d}/\mathbb{F}_q}(\eta))}.\]

\item If neither $\chi_s(R_{\mathbb{F}_{q^d}/\mathbb{F}_q}(\eta))$ nor $-\chi_s(R_{\mathbb{F}_{q^d}/\mathbb{F}_q}(\eta))$ are squares in $\mathbb{F}_q(\chi_s)^{\G}$ (they will still be squares in $\mathbb{F}_q(\chi_s)$), then \[v_s=\sqrt{-\chi_s(R_{\mathbb{F}_{q^d}/\mathbb{F}_q}(\eta))}.\]
\item If $\chi_s(R_{\mathbb{F}_{q^d}/\mathbb{F}_q}(\eta))$ is not a square in $\mathbb{F}_q(\chi_s)^{\G}$ but $-\chi_s(R_{\mathbb{F}_{q^d}/\mathbb{F}_q}(\eta))$ is a square in $\mathbb{F}_q(\chi_s)^{\G}$, then \[v_s=\frac{\sqrt{(n-1)\chi_s(R_{\mathbb{F}_{q^d}/\mathbb{F}_q}(\eta))}}{\sqrt{-n}}+\frac{\sqrt{-\chi_s(R_{\mathbb{F}_{q^d}/\mathbb{F}_q}(\eta))}}{\sqrt{-n}}\]\end{enumerate}
where $n$ is the smallest integer in the set $\{1,\ldots,p-1\}$ such that $-n$ is a square in $\mathbb{F}_p$.
%\[v_s=\left\{\begin{array}{ll}\sqrt{\chi_s(R_{\mathbb{F}_{q^d}/\mathbb{F}_q}(\eta))}& \text{\ \if $\chi_s(R_{\mathbb{F}_{q^d}/\mathbb{F}_q}(\eta))\in (\mathbb{F}_q(\chi_s)^{\G})^2$}
%\\\sqrt{-\chi_s(R_{\mathbb{F}_{q^d}/\mathbb{F}_q}(\eta))}& \text{\ \ if $\chi_s(R_{\mathbb{F}_{q^d}/\mathbb{F}_q}(\eta))\slashed{\in}(\mathbb{F}_q(\chi_s)^{\G})^2$ and $-\chi_s(R_{\mathbb{F}_{q^d}/\mathbb{F}_q}(\eta))\slashed{\in}(\mathbb{F}_q(\chi_s)^{\G})^2$}\\
 %\frac{\sqrt{(n-1)\chi_s(R_{\mathbb{F}_{q^d}/\mathbb{F}_q}(\eta))}}{\sqrt{-n}}+\frac{\sqrt{-\chi_s(R_{\mathbb{F}_{q^d}/\mathbb{F}_q}(\eta))}}{\sqrt{-n}}&\text{\ \ if $\chi_s(R_{\mathbb{F}_{q^d}/\mathbb{F}_q}(\eta))\slashed{\in}(\mathbb{F}_q(\chi_s)^{\G})^2$ and $-\chi_s(R_{\mathbb{F}_{q^d}/\mathbb{F}_q}(\eta)){\in}(\mathbb{F}_q(\chi_s)^{\G})^2$}\end{array}\right.\]
 \label{s=s'_lemma}
\end{lemma}

\begin{proof}
If $s=s'$, then $J\in\Gal(\mathbb{F}_q(\chi_s)/\mathbb{F}_q)$ and $\G$ is a subgroup of $\Gal(\mathbb{F}_q(\chi_s)/\mathbb{F}_q)$.
We remark that this means $[\mathbb{F}_q(\chi_s):\mathbb{F}_q]$ is even, which only happens when the $r$th roots of unity are not contained in $\mathbb{F}_q$,
i.e., $r\slashed{|}q-1$. We know that $\chi_s(R_{\mathbb{F}_{q^d}/\mathbb{F}_q}(\eta))\in \mathbb{F}_q(\chi_s)^{\G}$ and $|\G|=2$,
so from the properties of finite field extensions, $\chi_s(R_{\mathbb{F}_{q^d}/\mathbb{F}_q}(\eta))$ is a square in $\mathbb{F}_q(\chi_s)$.

For brevity we now write $a=\chi_s(R_{\mathbb{F}_{q^d}/\mathbb{F}_q}(\eta))$.
\begin{enumerate}
 \item  If $a\in (\mathbb{F}_q(\chi_s)^{\G})^2$ then $\sqrt{a}$
is fixed by $J$ and $\sqrt{a}.J({\sqrt{a}})=a$.
\item If $-a\ \slashed{\in}\ (\mathbb{F}_q(\chi_s)^{\G})^2$,
 then $J({\sqrt{-a}})=-\sqrt{-a}$ and $\sqrt{-a}.J({\sqrt{-a}})=--a=a$.
\item Finally, suppose that $a\ \slashed{\in}\ (\mathbb{F}_q(\chi_s)^{\G})^2$ and
$-a\in (\mathbb{F}_q(\chi_s)^{\G})^2$, then a basis for $\mathbb{F}_q(\chi_s)$ over
$\mathbb{F}_q(\chi_s)^{\G}$ is $\{\sqrt{a},\sqrt{-a}\}$ with $J(\sqrt{a})=-\sqrt{a}$ and $J(\sqrt{-a})=\sqrt{-a}$.

 Let $n$ be the
smallest integer in $\{1,\ldots,p-1\}$ such that $-n$ is a square in
$\mathbb{F}_p$ (note that $-(p-1)$ is always a square so $n$ is well defined). This means that $-(n-1)$ is not a square in $\mathbb{F}_p$. Note that since $a\slashed{\in}\ (\mathbb{F}_q(\chi)^{\G})^2$ and $-a{\in}\ (\mathbb{F}_q(\chi)^{\G})^2$ we must have $-a^2\ \slashed{\in}\ (\mathbb{F}_q(\chi)^{\G})^2$ and therefore $-1\ \slashed{\in}\ (\mathbb{F}_q(\chi)^{\G})^2$. We then know that $-1$ and $-(n-1)$ are not squares in $\mathbb{F}_p$, and so $(n-1)$ is a square in $\mathbb{F}_p$. Moreover, $J$ fixes $\sqrt{n-1}$ and $\sqrt{-n}$, therefore
\[J\left({\frac{\sqrt{n-1}}{\sqrt{-n}}.\sqrt{a}+\frac{1}{\sqrt{-n}}.\sqrt{-a}}\right)=\frac{\sqrt{n-1}}{\sqrt{-n}}.(-\sqrt{a})+\frac{1}{\sqrt{-n}}.\sqrt{-a},\]
and so
\[\left(\frac{\sqrt{n-1}}{\sqrt{-n}}.\sqrt{a}+\frac{1}{\sqrt{-n}}.\sqrt{-a}\right).J\left({\frac{\sqrt{n-1}}{\sqrt{-n}}.\sqrt{a}+\frac{1}{\sqrt{-n}}.\sqrt{-a}}\right)=\frac{1-n-1}{-n}a=a.\]

\end{enumerate}

\

\end{proof}

We remark that to calculate the image of an element under the
inverse of the decomposition map $D$, it is only required to solve a
set of $|S|$ simultaneous linear equations.

\begin{theorem}
 Let $v=D^{-1}(Dv)$ where $Dv=\prod_{s\in S}v_s$ with $v_s$ as described in Lemmas \ref{0_lemma}, \ref{sns'_lemma} and \ref{s=s'_lemma}. Let $\eta$ be a normal basis generator for $\mathbb{F}_{q^d}/\mathbb{F}_q$ constructed as above,
 then \[\eta'=v^{-1}\circ\eta\]
 is a self-dual normal basis generator for $\mathbb{F}_{q^d}$ over $\mathbb{F}_q$.\label{finite_p'_theorem}
\end{theorem}

\begin{proof}
 This now follows directly from Theorem \ref{construction_theorem} and Lemmas \ref{0_lemma}, \ref{sns'_lemma} and \ref{s=s'_lemma}.

\

\end{proof}

\begin{remark} We originally proved Theorem \ref{construction_theorem} only for $\mathbb{F}_q$ with characteristic not equal to $2$. However, this construction is still valid when $p=2$ and $2\slashed{|}d$, so we have now constructed self-dual normal bases for finite extensions of finite fields whenever they exist.
\end{remark}

\section{Local Fields}\label{local}
Throughout this section we will always assume $p$ to be an odd prime and $K$ to be a finite extension of $\Q_p$. We fix an algebraic closure of $K$, denoted  $\bar{K}$, and let $L/K$ be a finite Galois extension inside $\bar{K}$ with abelian Galois group $G$. Let $\bo_K$ (resp. $\bo_L$) be the valuation ring of $K$ (resp. $L$) and let $l$ and $k$ be the residue fields of $L$ and $K$ respectively. Recall that we define $A_{L/K}$ to be the unique fractional $\bo_L$-ideal such that \[A_{L/K}=\D_{L/K}^{-1/2}.\]

From \cite{Fainsilber+Morales} Theorem 4.5 and \cite{Erez2} Theorem 1, we know that a self-dual integral normal basis for $A_{L/K}$ exists if and only if $L/K$ is at most weakly ramified (this result is valid even if $G$ is not abelian). In this section we describe how to construct such bases whenever they exist, providing $G$ is abelian and that the residue field of $K$ is not of characteristic $2$.

\subsection{Reduction of the Problem}

We begin by proving a series of results that will reduce the situation to finding self-dual integral normal bases for $A_{L/K}$ when $L/K$ is either totally ramified or unramified.

\begin{proposition}
 With $L/K$ as above, there exist fields $L',L^{un},L^{tot}\subseteq\bar{K}$ such that 
\begin{enumerate}
\item $L'/L$ is an unramified Galois extension of fields, 
\item $L'=L^{un}L^{tot}$,
\item $L^{tot}/K$ (resp. $L^{un}/K$) is a totally ramified (resp. unramified) Galois extension of fields.\end{enumerate}\label{L_decomposition_prop}
\end{proposition}

\begin{proof}
The fact that $L$ is contained in the compositum of a totally ramified and an unramified extension of $K$ is a direct result of the construction of the maximal abelian extension of $K$ in local class field theory. See, for example \cite{iwasawa} \S7.4. We give here a constructive proof that ensures that this compositum is an unramified extension of $L$.

Let $m$ denote the exponent of $G=\Gal(L/K)$. Let $G_0$ be the
decomposition group of $G$ and let $n=|G/G_0|$. Then the field fixed
by $G_0$ will be the unique unramified extension of $K$ of degree
$n$, denoted $K_{un}^n$, and will be the maximal unramified
extension of $K$ contained in $L$. We note that
$\Gal(K_{un}^n/K)\cong C_n$ and that $n|m$.

We now consider the unique unramified extension of $K$ of degree
$m^2$ in $\bar{K}$. We observe that $\Gal(K_{un}^{m^2}/K)\cong C_{m^2}$ and
$\Gal(K_{un}^{m^2}/K_{un}^n)\cong C_{m^2/n}$. We let
$L'=LK_{un}^{m^2}$ be the compositum of $L$ and $K_{un}^{m^2}$ inside $\bar{K}$ and
let $G'=\Gal(L'/K)$. We observe that $L'/L$ is unramified of degree
$m^2/n$ and that $G'_0\cong G_0$. We have $G'/G'_0\cong\Gal(K_{un}^{m^2}/K)\cong
C_{m^2}$ and that $|G'_0|=|G_0|$ divides $m$. Let $a\in G'$ be any element that maps onto a generator of $\Gal(K_{un}^{m^2}/K)$ when we quotient $G'$ by $G'_0$, then $<a>$ is a subgroup of $G'$ isomorphic to
$C_{m^2}$. We let $F$ be the field fixed by one such subgroup, see Fig. \ref{L'} for details.

\begin{figure}[ht]\hspace{1.5cm}\xymatrix{&L'\ar@{-}@/0pc/[dl]\ar@{-}@/l1pc/[dl]_{G'_0\cong
G_0}\ar@{-}@/0pc/[ddrr]\ar@{-}@/0pc/[ddrrrr]\ar@{-}@/r3pc/[ddrrrr]^{C_{m^2}}
\\K_{un}^{m^2}\ar@{-}@/l3pc/[dddrrr]_{C_{m^2}}\ar@{-}@/0pc/[ddrr]\ar@{-}@/r1pc/[ddrr]^{C_{m^2/n}}&&&&&&
\\&&&L\ar@{-}@/r1pc/[dd]^G\ar@{-}@/0pc/[dd]\ar@{-}@/0pc/[dl]\ar@{-}@/u0.5pc/[dl]_{G_0}&&F\ar@{-}@/0pc/[ddll]
\\&&K_{un}^n\ar@{-}@/r0.5pc/[dr]^{C_n}\ar@{-}@/0pc/[dr]\\&&&K}\caption{The extension $L'/K$ and intermediate fields}\label{L'}\end{figure}
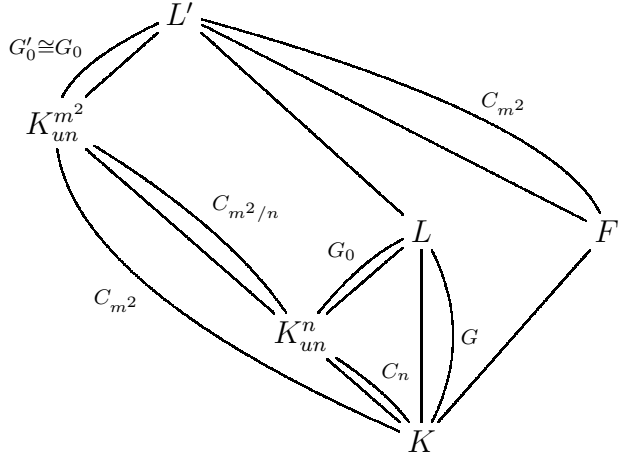

We now claim that $F\cap K_{un}^{m^2}=K$. Indeed, suppose not and
that there exists an integer $s|m^2$ with $F\cap
K_{un}^{m^2}=K_{un}^{s}$. 
\begin{itemize}\item If $s\leq n$, then $K_{un}^s\subseteq K_{un}^n\subseteq L$. In this case we let $H$ be the subgroup of $G$ that fixes $K_{un}^s$ and note that the exponent of $H$ is less than or equal to $m$ (as $m$ is the exponent of $G$).
\item If $s>n$, then $K_{un}^s\slashed{\subset} L$, but $K_{un}^s\subseteq K_{un}^sL$. In this case we let $H=\Gal(K_{un}^sL/K_{un}^s)\cong G_0$ and again note that the exponent of $H$ is less than or equal to $m$.\end{itemize}
With $H$ defined like this we must then have $\Gal(L'/K_{un}^{s})\leq
C_{m^2/s}\times H$. The exponent of $\Gal(L'/K_{un}^{s})$
is then at most $lcm(|H|,|C_{m^2/s}|)$, and so is at most $lcm(m,m^2/s)$. However, as $K_{un}^{s}\subseteq F$ we must have
$C_{m^2}$ as a subgroup of $\Gal(L'/K_{un}^{s})$, and so the
exponent of $\Gal(L'/K_{un}^{s})$ must be at least $m^2$. Therefore
the exponent of $\Gal(L'/K_{un}^{s})$ must be $m^2$ which means $s=1$, and so $F\cap K_{un}^{m^2}=K$. This in turn means that $F/K$ is
totally ramified, that $L'=FK_{un}^{m^2}$ and that $G'\cong
G'_0\times C_{m^2}$.

Finally we observe that $L'=L_{un}^{m^2/n}$ is the unique unramified extension of $L$ of degree $m^2/n$ contained in $\bar{K}$. \end{proof}

\begin{lemma} Let $L_1/K$ and $L_2/K$ be abelian Galois extensions in $\bar{K}$ such that $L_1/K$ is unramified, $L_1\cap L_2=K$ and $L_1L_2=L$. Let $x_1$ (resp. $x_2$) be a self-dual integral normal basis generator for $A_{L_1/K}$ (resp. $A_{L_2/K}$). Then $x_1x_2$ is a self-dual integral normal basis generator for $A_{L/K}$.\label{unram_mult_lemma} \end{lemma}

\begin{proof}From Lemma \ref{mult_lemma} we know that $Tr_{L/K}(x_1x_2,(x_{1}x_{2})^g)=\delta_{1,g}$ for all $g\in \Gal(L/K)$.
From \cite{Frohlich-Taylor} III.2.15 we know that $\D_{L/K}=\D_{L/L_2}\D_{L_2/K}$ and hence $A_{L/K}=A_{L/L_2}A_{L_2/K}$. We know $L/L_2$ is unramified, so $A_{L/L_2}=\bo_L$. Therefore, $v_L(A_{L/K})=v_L(A_{L_2/K}\bo_L)=v_{L_2}(A_{L_2/K})$. The element $x_{1}$ must be a unit in $\bo_{L_1}$ so $v_L(x_{1}x_2)=v_L(A_{L/K})$, hence $x_{1}x_{2}\in A_{L/K}$. The result now follows from \cite{Pickett} Lemma 8.\end{proof}

 \begin{lemma} Let $L/K$ be an abelian extension with Galois group $G$. Let $H$ be a subgroup of $G$. For $x\in L$, if $Tr_{G}(x,g(x))=\delta_{1,g}$ for $g\in G$, then
$Tr_{(G/H)}(Tr_{H}(x),Tr_{H}{\tilde{g}}(x))=\delta_{1,\tilde{g}}$ for $\tilde{g}\in G/H$. \label{trace_down_lemma}\end{lemma}

\begin{proof}
\[\begin{array}{ll}Tr_{(G/H)}(Tr_{H}(x),{\tilde{g}}(Tr_{H}(x)))
&=Tr_{(G/H)}\left(\left(\displaystyle \sum_{h_1\in H}{h_1}(x)\right){\tilde{g}}\left(\displaystyle\sum_{h_2\in H}{h_2}(x)\right)\right)\\
 &=Tr_{(G/H)}\left(\left(\displaystyle \sum_{h_1\in H}{h_1}(x)\right)\left(\displaystyle\sum_{h_2\in H}{\tilde{g}h_2}(x)\right)\right)\\
  &=Tr_{(G/H)}\left(\displaystyle \sum_{h_1\in H}\sum_{h_2\in H}{h_1}(x){\tilde{g}h_2}(x)\right) \\
  &=Tr_{(G/H)}\left(\displaystyle \sum_{h_1\in H}\sum_{h\in H}{h_1}(x){\tilde{g}hh_1}(x)\right)\text{\ \ where $h_2h_1^{-1}=h$}  \\ 
&=Tr_{(G/H)}\left(\displaystyle \sum_{h_1\in H}\sum_{h\in H}{h_1}(x){h_1\tilde{g}h}(x)\right) \text{\ \ as $G$ is abelian}\\
   &=Tr_{(G/H)}\left(\displaystyle \sum_{h_1\in H}\sum_{h\in H}{h_1}(x.{\tilde{g}h}(x))\right)\\
&=Tr_{(G/H)}\left(\displaystyle Tr_H\left(\sum_{h\in H}x.{\tilde{g}h}(x)\right)\right)\\
  &=Tr_G\left(\displaystyle \sum_{h\in H}x.{\tilde{g}h}(x)\right)\end{array}.\]
By hypothesis $Tr_G(x.{\tilde{g}h}(x))=0$ unless $\tilde{g}=h^{-1}$. The only time this can happen with $h\in H$ and $\tilde{g}\in G/H$ is when $\tilde{g}=1$ and $h=1$. Therefore, $Tr_G\left(\sum_{h\in H}x.{\tilde{g}}h(x)\right)=\delta_{1,\tilde{g}}$ which proves the result.\end{proof}

\begin{theorem} Let $L'$, $L^{un}$ and $L^{tot}$ be as in Proposition \ref{L_decomposition_prop} and $L/K$ be as set out at the beginning of this section. Let $x_{un}$ (resp. $x_{tot}$) be a self-dual integral normal basis generator for $A_{L^{un}/K}$ (resp. $A_{L^{tot}/K}$), then \[Tr_{L'/L}(x_{un}x_{tot})\] is a self-dual integral normal basis generator for $A_{L/K}$. \label{product_theorem}
\end{theorem}
\begin{proof}
 From Lemma \ref{unram_mult_lemma} we know that $x_{un}x_{tot}$ will be a self-dual integral normal basis for $A_{L'/K}$. From Lemma \ref{trace_down_lemma} we know that $Tr_{L'/L}(x_{un}x_{tot})$ will be a self-dual element of $L$, and so using \cite{Pickett} Lemma 8 we just need to show that $Tr_{L'/L}(x_{un}x_{tot})\in A_{L/K}$. It is therefore sufficient to show that $Tr_{L'/L}(A_{L'/K})\subseteq A_{L/K}$.

The extension $L'/L$ is unramified so we know that $\D_{L'/L}=\bo_{L'}$. From \cite{Serre} III \S3 Proposition 7, we therefore have $Tr_{L'/L}(A_{L'/K})\subseteq A_{L/K}$ if and only if $A_{L'/K}\subseteq A_{L/K}\bo_{L'}$. However, as $L'/L$ is unramified we have $v_{L'}(\D_{L'/K})=v_L(\D_{L/K})$, which implies $v_{L'}(A_{L'/K})=v_L(A_{L/K})=v_{L'}(A_{L/K}\bo_{L'})$. This now implies  $A_{L'/K}=A_{L/K}\bo_{L'}$, which proves the result.
\end{proof}

With $p$ an odd prime, $K$ a finite extension of $\Q_p$ and $L/K$ any finite weakly ramified abelian extension of odd degree we can using Proposition \ref{L_decomposition_prop} and Theorem \ref{product_theorem} to reduce the problem of finding a self-dual integral normal basis for $A_{L/K}$ to finding self-dual integral normal bases for $A_{L^{tot}/K}$ and $A_{L^{un}/K}$. We therefore continue considering the totally ramified and the unramified cases separately.

\subsection{Totally Ramified Extensions}
Any totally ramified abelian extension of $K$ will be contained in some division field arising from a Lubin-Tate formal group. We will give a brief overview of the definitions of such division fields and some of their properties. For a more  detailed account see, for example, \cite{iwasawa} or \cite{serre-lubintate}.

Recall that $K$ is a finite extension of $\Q_p$ and $\bar{K}$ is a fixed algebraic closure of $K$. Let $\pi$ be a uniformising
parameter for $\bo_K$ and let $q=|\bo_K/\bp_K|$ be the cardinality
of the residue field. We let $f(X)\in X\bo_K[[X]]$ be such that
\begin{eqnarray} \nonumber f(X)&\equiv&\pi X\mod \deg 2\txt{,\ \ \
\ \ and\ \ \ \ \ \ } f(X)\equiv X^q\mod\pi.\end{eqnarray}

There then exists a unique formal group $F_f(X,Y)\in\bo_K[[X,Y]]$
which admits $f$ as an endomorphism, i.e., $F_f\circ f=f\circ F_f$ and $F_f(X,Y)$ satisfies some
identities that correspond to the usual group axioms, see
\cite{serre-lubintate} \S3.2 for full details. For $a\in \bo_K$,
there exists a unique formal power series, $[a]_f(X)\in
X\bo_K[[X]]$, that commutes with $f$ and such that $[a]_f(X)\equiv
aX\mod \deg 2$. We note that these conditions ensure that $[\pi]_f(X)=f(X)$.

We can use the formal group, $F_f$, and the formal
power series, $[a]_f$, to define an $\bo_K$-module structure on
$\bp_{\bar{K}}^c=\bigcup_{L}\bp_{L}$, where the union is taken over
all finite Galois extensions $L/K$ where $L\subseteq \bar{K}$. We
let $E_{f,n}=\{x\in\bp_{\bar{K}}^c:[\pi^n]_f(x)=0\} $ and $E_f=\bigcup_n E_{f,n}$ and then define $K_{\pi,n}=K(E_{f,n})$ and $K_{\pi}=K(E_f)$. We remark that the sets $E_{f,n}$ and $E_f$ depend on
the choice of the polynomial $f$ but due to a property of the formal
group (see \cite{serre-lubintate} \S3.3 Prop. 4), $K_{\pi,n}$ and $K_{\pi}$
depend only on the uniformising parameter $\pi$. The extensions
$K_{\pi,n}/K$ and $K_{\pi}/K$ are totally ramified abelian extensions, $[K_{\pi,n}:K]=q^{n-1}(q-1)$ and every totally ramified abelian extension of $K$ in $\bar{K}$ is contained in $K_{\pi}$ for some $\pi$. When using Lubin-Tate division fields we shall always assume that $f(X)=[\pi](X)=X^q+\pi X$.

In this section we give different constructions for self-dual integral normal basis generators for $A_{L/K}$ depending on whether $L/K$ is tamely or wildly ramified.

\subsubsection{Totally Tamely Ramified Extensions}

 Let $L/K$ be a finite totally tamely ramified abelian extension with $[L:K]=d$ and $d$ odd. We know that $L$ must be contained in $K_{\pi,n}$ for some uniformising parameter $\pi$ of $\bo_K$ and some integer $n$. The degree $[K_{\pi,m+r}:K_{\pi,m}]$ is a power of $p$ for all $m,r\in\mathbb{N}$ and $m>1$, therefore $L\subseteq K_{\pi,1}$. As described above, we let $f(X)=[\pi](X)=X^q+\pi X$ and $K_{\pi,1}/K$ is the splitting field for the polynomial $X^{q-1}+\pi$. The $(q-1)$th roots of unity are contained in $K$ (see \cite{Frohlich-Taylor}, Theorem 25) and $d|(q-1)$, therefore $L/K$ must be a Kummer extension. Letting $\pi=-\tau$ we see that the field $L$ will then be the subfield of $K_{\pi,1}$ that splits the polynomial $X^d-\tau$. We let $\tau^{1/d}$ be a root of $X^d-\tau$, then $L=K(\tau^{1/d})$.

\begin{lemma} $Tr_{L/K}(\tau^{i/d})=0$ for all $d\slashed{|}i$. \label{Tr(xi)=0}\end{lemma}
\begin{proof}The minimal polynomial of $\tau^{i/d}$ over $K$ is $X^{d/\gcd(i,d)}-\tau^{i/\gcd(i,d)}$, therefore $Tr_{L/K}(\tau^{i/d})=0$ for all $d\slashed{|}i$.
\end{proof}

\begin{lemma} Let
\[\begin{array}{ll}
x&=\displaystyle\frac{\tau^{(1-d)/2d}+\tau^{(3-d)/2d}+\ldots+\tau^{-1/d}+1+\tau^{1/d}+\ldots+\tau^{(d-1)/2d}}{d}\\
&=\displaystyle\frac{1}{d}\left(\frac{1-\tau}{1-\tau^{1/d}}\right)\tau^{(1-d)/2d}\end{array}\]
then $Tr_{L/K}(x,g(x))=\delta_{1,g}$. \label{kummerlemma}
\end{lemma}
\begin{proof}  We have $x=b_0
+b_1\tau^{1/d}+b_2\tau^{2/d}+\ldots+b_{d-1}\tau^{(d-1)/d}$, with
\[b_0,b_1,\ldots,b_{(d-1)/2}=1/d\text{\ \ \ and\ \ \ } b_{(d+1)/2},b_{(d+3)/2},\ldots,b_{d-1}=1/\tau d.\]

Let $g_k:\tau^{1/d}\mapsto\zeta^k\tau^{1/d}$ and $x_k={g_k}(x)$. We see that
\[x_k=\sum_{i=0}^{d-1}b_i\zeta^{ik}\tau^{i/d},\] and so \[xx_k=\sum_{i=0}^{d-1}\sum_{j=0}^{d-1}b_ib_j\zeta^{ik}\tau^{(i+j)/d}.\]
From Lemma \ref{Tr(xi)=0}, we know that
\[Tr_{L/K}(b_ib_j\zeta^{ik}\tau^{(i+j)/d})=(b_ib_j\zeta^{ik})Tr_{L/K}(\tau^{(i+j)/d})=0\] unless $d|(i+j)$.
Therefore,
\begin{equation}Tr_{L/K}(xx_k)=Tr_{L/K}\left(\sum_{i=0}^{d-1}b_ib_{d-i}\zeta^{ik}\tau\right). \label{trace equation}\end{equation}
Substituting the above values of $b_i$ into Equation (\ref{trace
equation}) we get
\[ Tr_{L/K}(xx_k)=Tr_{L/K}\left(\sum\limits_{i=0}^{d-1}(1/d).(1/\tau d).\zeta^{ik}.\tau\right)=Tr_{L/K}\left(\sum\limits_{i=0}^{d-1}\zeta^{ik}/d^2\right).\] When $k=0$,
 we have $(\sum_{i=0}^{d-1}\zeta^{ik}/d^2)=1/d$,
and so $T_{L/K}(x,x)=1$ as required. When $k\neq 0$, we have $(\sum_{i=0}^{d-1}\zeta^{ik}/d^2)=0$, and so
$T_{L/K}(x,{g}(x))=\delta_{1,g}$ as required.

\

\end{proof}

\begin{theorem}The element \[x=\displaystyle\frac{1}{d}\left(\frac{1-\tau}{1-\tau^{1/d}}\right)\tau^{(1-d)/2d}\] is a self-dual integral normal basis generator for $A_{L/K}$.\end{theorem}
\begin{proof}The extension $L/K$ is tamely ramified, therefore $\gcd(d,p)=1$ and $d$ is a unit in $\bo_L$. The polynomial $X^d-\tau$ is Eisenstein so $\tau^{{1/d}}$ is a uniformising parameter in $L$. Therefore $(1-\tau)$ and $(1-\tau^{1/d})$ are both units, and so $v_L(x)=(1-d)/2$.

Using the lower numberring of ramification groups, we have $G=G_0$ and $G_1=0$ as $L/K$ is totally tamely ramified. From Hilbert's formula for the valuation of the different (see \cite{Serre} IV \S2 Prop. 4), we then know that $v_L(\D_{L/K})=d-1$ and $v_L(A_{L/K})=(1-d)/2$. Therefore $x\in A_{L/K}$. From Lemma \ref{kummerlemma} we know that $Tr_{L/K}(x,g(x))=\delta_{1,g}$. Our result now follows from \cite{Pickett} Lemma 8.

\end{proof}

\begin{remark}
We remark that the construction in this section is still valid when $p=2$.
\end{remark}

\subsubsection{Totally Wildly Ramified Extensions}

We now let $L/K$ be a totally wildly weakly ramified abelian extension of odd degree with Galois group $G$. From \cite{Byott-Int_Gal_Mod_Struc_Some_LubinTate} \S4 Lemma 4.2 we know $L$ must be contained in $K_{\pi,2}$ for some uniformising parameter $\pi$ and that $[L:K]=p^i$ for some integer $i$. This lemma also shows that, with the lower numbering of ramification groups, $G=G_0=G_1$ and $G_2=\{1\}$. Our field $L$ will then be contained in the unique field $M\subset K_{\pi,2}$ such that $[M:K]=q$. We let $\Gamma=\Gal(M/K)$. See Fig. \ref{M/K}.

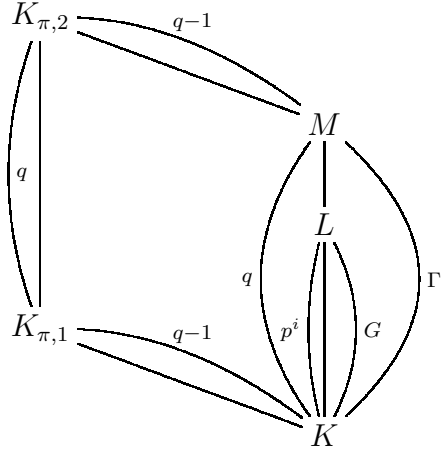
\begin{figure}[ht]\[\xymatrix{K_{\pi,2}\ar@{-}@/l1pc/[ddd]^{q}\ar@{-}@/u1pc/[drrr]^{q-1}\ar@{-}@/l0pc/[drrr]\ar@{-}@/l0pc/[ddd]
\\&&&M\ar@{-}@/l2pc/[ddd]_{q}\ar@{-}@/r3pc/[ddd]^{\Gamma}\ar@{-}@/l0pc/[d]\\&&&L\ar@{-}@/l0pc/[dd]\ar@{-}@/l0.5pc/[dd]_{p^i}\ar@{-}@/r1pc/[dd]^{G}\\K_{\pi,1}\ar@{-}@/u1pc/[drrr]^{q-1}\ar@{-}@/l0pc/[drrr]\\&&&K\\}\]\caption{Weakly ramified extensions of $K$}\label{M/K}\end{figure}

\begin{proposition}Let $L/K$ be an odd abelian $p$-extension with Galois group $G$ which is at most weakly ramified. Let $x\in L$ be an integral normal basis generator for $A_{L/K}$. Then $R_{L/K}(x)$ is a square in $\bo_K[G]^{\times}$ and using Notation \ref{sqrt_notation}, \[\frac{1}{\sqrt{R_{L/K}(x)}}\circ x\] is a self-dual integral normal basis generator for $A_{L/K}$. \label{add_weak_prop} \end{proposition}

\begin{proof}
We essentially follow the same method in the proof of Theorem \ref{finite_p_theorem}.
The group $ G$ is a $p$-group, therefore from Lemma \ref{exact sequence} we know that the following sequence is exact
\begin{equation}1\rightarrow(1+\J)\rightarrow(\bo_K
 [G])^{\times}\stackrel{\theta}{\rightarrow} k^{\times}\rightarrow
1.\label{methodexactsequence}\end{equation} where $\J$ is the Jacobson radical of $\bo_K [G]$ and $k$ is the residue field of $K$. The homomorphism $\theta:(\bo_K [G])^{\times}\rightarrow k^{\times}$ is defined as
\[\theta:\sum_{ g\in  G}a_{ g} g\longmapsto (\sum_{ g\in  G}a_{ g})
\mod{\pi}\]for $\sum_{ g\in  G}a_{ g} g\in(\bo_K
 [G])^{\times}$. We now observe that $k^{\times}\subseteq(\bo_K
 [G])^{\times}$ so there exists a natural embedding,
$j:k^{\times}\rightarrow(\bo_K [G])^{\times}$ which has the property $\theta
j=1_{k^{\times}}$. Therefore the short exact
sequence, (\ref{methodexactsequence}), is split and
\[(\bo_K [G])^{\times}=k^{\times}\times(1+\J).\]

Let $\varepsilon:\bo_K [G]\rightarrow \bo_K$ be the augmentation map on $\bo_K [G]$. A straightforward calculation, left as an exercise, shows that $\varepsilon(R_{L/K}(x))=Tr_{ G}(x)^2$, and so $\theta(R_{L/K}(x))=\varepsilon(R_{L/K}(x))\mod\pi$ is a square in $k^{\times}$. The group $(1+\J)$ is a pro-$p$ group with
$(p,2)=1$, so every element of $(1+\J)$ will be a square. Therefore,
we can write $R_{L/K}(x)=w^2$ for some $w\in(\bo_K [G])^{\times}$.

From Theorem \ref{construction_theorem} we know that $R_{L/K}(x)=u\bar{u}$ for some $u\in(\bo_K [G])^{\times}$.
We then have $w^2=u\overline{u}$ and $\overline{w}^2=\overline{w^2}=\overline{u\overline{u}}=u\overline{u}$, therefore $w^2=\overline{w}^2$. Since $w\in(\bo_K[G])^{\times}$ we can write $w=(w_1,w_2)$ and $w^2=(w_1^2,w_2^2)$ with
$w_1\in k^{\times}$ and $w_2\in(1+\J)$. We then see that
$\overline{w}=(w_1,\overline{w}_2)$ and $\overline{w}^2=(w_1^2,\overline{w}_2^2)$. The group $(1+\J)$ is a
pro-$p$ group with $(p,2)=1$ so $\psi_2:a\mapsto a^2$ is an
automorphism of $(1+\J)$. Therefore, if $w_2^2=\overline{w}_2^2$
then $w_2=\overline{w}_2$, and so $w=\overline{w}$.

If we let $w=\sqrt{R_{L/K}(x)}$ the result then follows from Theorem \ref{construction_theorem}
\end{proof}

\begin{theorem}
Let $f(X)=X^q+\pi X$ and $\alpha$ be a primitive root of $[\pi^2]_f(X)$. Let $x=\alpha^{q-1}/p$, using Notation \ref{sqrt_notation}, we have that \begin{enumerate} \item $\displaystyle Tr_{M/L}\left(\frac{1}{\sqrt{R_{M/K}(x)}}\circ x\right)$ and \item $\displaystyle\frac{1}{\sqrt{R_{L/K}(Tr_{M/L}(x))}}\circ Tr_{M/L}(x)$ \end{enumerate}  are both self-dual integral normal basis generators for $A_{L/K}$. \end{theorem}
\begin{proof}

In \cite{Vinatier-3} Corollary 2.5, Vinatier proves that any $x\in M$ with $v_M(x)=v_M(A_{M/K})=1-q$ is a normal basis generator for $A_{M/K}$ and that $Tr_{M/L}(x)$ is a normal basis generator for $A_{L/K}$. Therefore, if $v_M(x)=1-q$, both parts follow from Proposition \ref{add_weak_prop} and Lemma \ref{trace_down_lemma}.

We now just need to show that $v(\alpha^{q-1}/p)=1-q$. We know $\alpha$ is a uniformising parameter for $\bo_{K_{\pi,2}}$ as $[\pi^2](X)/[\pi](X)$ is Eisenstein. The extension $K_{\pi,2}/M$ is totally ramified, so $N_{K_{\pi,2}/M}(\alpha)$ will be a uniformising parameter for $\bo_M$. With the above choice of $f(X)$, we clearly have $[\zeta]_f(X)=\zeta X$ for $\zeta^{q-1}=1$. We then see that the group $\{[\zeta]_f(\alpha): \zeta^{q-1}=1\}$ is of order $(q-1)$ and each element is a distinct root of $[\pi^2](X)$, and so these must be the Galois conjugates of $\alpha$ for $K_{\pi,2}/M$. Therefore $N_{K_{\pi,2}/M}(\alpha)=\alpha^{q-1}$. Our result then follows as $v_M(p)=q$.

\end{proof}

\begin{remark}
\begin{enumerate}
%\item  If $\alpha$ is a primitive $[\pi^2]$-division point then each element of $\Gal(K_{\pi,2}/K_{\pi,1})$ maps $\alpha$ onto $[1+u\pi](\alpha)$ for some $u\in k$. Therefore, for $\gamma\in\Gamma$, we have $\gamma(\alpha^{q-1}/p)=([1+u\pi](\alpha))^{q-1}/p$ for some $u\in k$. The proof of the existence $[1+u\pi](X)$ is constructive, see \cite{serre-lubintate} \S3, and so once $1/\sqrt{R_{M/K}(x)}$ has been calculated with $x=\alpha^{q-1}/p$, the construction above really is explicit.
\item When $p=2$ there will not exist any abelian totally wildly weakly ramified extensions of $K$ of odd degree. \end{enumerate}
\end{remark}

\subsection{Unramified extensions}
In this section we assume that $L/K$ is a finite abelian unramified extension of odd degree, $d$. Let $k$ (resp. $l$)
be the residue field of $\bo_K$ (resp. $\bo_L$). Let $|k|=q$, which implies that $|l|=q^d$. Let $\zeta\in\mu_{q^n-1}$
be a primitive $(q^n-1)$th root of unity in $\bar{K}$. From \cite{Frohlich-Taylor} Theorem 25 we know that $L=K(\zeta)$,
 $\bo_L=\bo_K[\zeta]$ and that $\Gal(L/K)=G$ is generated by the Frobenius element, $\phi_K:\zeta\mapsto \zeta^q$.

The group $G=\Gal(L/K)$ is a finite cyclic group, and so it can be written as
a finite product \[G=\prod_{i}C_{n_i}\] where $C_{n_i}$ is a cyclic
group of order $n_i$ such that $n_i$ is a power of a unique prime for all $i$.
We can then write $L$ as a compositum of fields $L=\prod_iL_i$ where
$\Gal(L_i/K)\cong C_{n_i}$, $L_i/K$ is unramified and $L_i\cap
L_j=K$ for all $i\neq j$. Using Lemma \ref{unram_mult_lemma}, we
are now reduced to considering the case where $L/K$ is cyclic and $d$
is some prime power.

Considering the extension of residue fields we will make crucial use of the basis generators constructed for extensions of finite fields. 

\subsubsection{Unramified $p$-extensions}
We first consider the case where $K$ is a finite extension of $\Q_p$ and $d$ is a power of $p$.

\begin{theorem} Let $K$ be a finite extension of $\Q_p$ with $p\neq 2$.
Let $L/K$ be a finite unramified abelian extension with Galois group $G$ and
$[L:K]$ a power of $p$. Let $x\in\bo_L$ such that
$x\equiv\eta\mod\bp_L$ where $\eta$ is the normal basis generator
for $l/k$ constructed as in Section \ref{p_finite_section} and with Notation \ref{sqrt_notation}, then \[\frac{1}{\sqrt{R_{L/K}(x)}}\circ
x\] is a self-dual integral normal basis generator for
$\bo_L=A_{L/K}$.
\end{theorem}

\begin{proof}
 Let $\pi$ be a uniformising parameter of $\bo_K$, then the Jacobson radical of $\bo_K$ is $\pi\bo_K$. By assumption we have $l= k[{G}].\eta$ which implies that $\bo_L\cong \bo_K[{G}].\eta\mod\pi\bo_L$. Therefore, by Nakayama's lemma (see \cite{Atiyah-MacDonald}, Ch.2 Prop. 2.6), we have $x\in\bo_L$ is an integral normal basis generator for $\bo_L$ if $x\equiv \eta\mod\pi\bo_L$. Our result now follows from Proposition \ref{add_weak_prop}.
\end{proof}

\subsubsection{Unramified $p'$-extensions}

\begin{theorem} For $K$ a finite extension of $\Q_p$ where $p\neq 2$
and for $L/K$ a finite abelian unramified extension of degree
$d=r^i$ where $r$ is an odd prime different to $p$.
 Let $x\in \bo_L$ be such that $x\equiv\eta'$ where $\eta'$ is as in Theorem \ref{finite_p'_theorem} and let $\sqrt{R_{L/K}(x)}$ be as in Notation \ref{sqrt_notation}, then \[\frac{1}{\sqrt{R_{L/K}(x)}}\circ x\] is a self-dual integral normal basis for $\bo_L=A_{L/K}$.
\end{theorem}

\begin{proof}
Let $\pi$ be a uniformising parameter for $\bo_K$. We have $x\equiv \eta'\mod\pi\bo_L$, which means
 \[\begin{array}{ll}Tr_G(x,g(x))&\equiv Tr_G(\eta',g(\eta'))\mod\pi\bo_K\\&\equiv\delta_{1,g}\mod\pi\bo_K.\end{array}\]
 Therefore, we have $R_{L/K}(x)=\sum_{g\in G}(x,g(x))g\in1+\pi\bo_K [G]$. This means that $R_{L/K}(x)\in 1+\J$ where $\J$ is the Jacobson radical of $\bo_K [G]$.
  We know that $1+\J$ is a pro-$p$ group, with $(p,2)=1$, therefore every element is a square.
  We then have $R_{L/K}(x)=\sqrt{R_{L/K}(x)}\sqrt{R_{L/K}(x)}$ and $\sqrt{R_{L/K}(x)}=\overline{\sqrt{R_{L/K}(x)}}$.
   The result now follows from the arguments in Theorem \ref{construction_theorem}.
\end{proof}

\begin{remark} \begin{enumerate}\item An observation of Vinatier is that it is seems very attractive to use the square-root of the inverse of an element to construct basis generators for the square-root of the inverse different.
\item The theory presented here is complete for finite fields and local fields $L/K/\Q_p$, with $p\neq 2$. When $p=2$ the only case needed to complete the theory is $L/K$ unramified and $[L:K]=r^i$ for some odd prime $r$ and $i\in\mathbb{N}$.\end{enumerate}
\end{remark}

\bibliography{bib}
\end{document}